\documentclass{article}

\usepackage{amsmath,amsfonts,amssymb,amsthm,tikz}

\usepackage[utf8]{inputenc}
\usepackage{cite}

\newcommand{\norm}{{\mathcal{N}}}
\newcommand{\calO}{{\mathcal{O}}}
\newcommand{\calP}{{\mathcal{P}}}

\newcommand{\C}{{\mathbb{C}}}
\newcommand{\F}{{\mathbb{F}}}

\newcommand{\Q}{{\mathbb{Q}}}
\newcommand{\Z}{{\mathbb{Z}}}

\newcommand{\height}{\mathrm{h}}

\newcommand{\gerp}{{\mathfrak{p}}}

\DeclareMathOperator{\logast}{log^\ast\!}

\newcommand{\ph}{\varphi}

    \makeatletter
    \let\@fnsymbol\@alph
    \makeatother

\title{Big prime factors in orders of elliptic curves over finite fields}
\author{}

\newtheorem{theorem}{Theorem}[section]
\newtheorem{proposition}[theorem]{Proposition}

\newtheorem{remark}[theorem]{Remark}

\numberwithin{equation}{section}

\makeatletter

\renewcommand*\l@section[2]{%
  \ifnum \c@tocdepth >\z@
    \addpenalty\@secpenalty
    \addvspace{0.2em \@plus\p@}%
    \setlength\@tempdima{1.5em}%
    \begingroup
      \parindent \z@ \rightskip \@pnumwidth
      \parfillskip -\@pnumwidth
      \leavevmode \bfseries
      \advance\leftskip\@tempdima
      \hskip -\leftskip
      #1\nobreak\hfil \nobreak\hb@xt@\@pnumwidth{\hss #2}\par
    \endgroup
  \fi}

\makeatother

\newcounter{fn}

\begin{document}

\hfuzz 4.3pt

%\author{ Yuri Bilu, Haojie Hong}

\author{Yuri Bilu\footnote{Supported  by the ANR project JINVARIANT}, \setcounter{fn}{\value{footnote}} \  Haojie Hong\footnote{Supported by the China Scholarship Council grant CSC202008310189} \  and Florian Luca\footnotemark[\value{fn}]}

\maketitle

\begin{abstract}
Let $E$ be an elliptic curve over the finite field $\mathbb F_q$. We prove that, when $n$ is a sufficiently large positive integer, $\#E(\mathbb F_{q^n})$ has a prime factor exceeding $n\exp(c\log n/\log\log n)$.
\end{abstract}

{\footnotesize

\tableofcontents

}

\section{Introduction}
A Lucas sequence $(u_n)_{n\ge 0}$ is a binary recurrent sequence of integers satisfying ${u_{n+2}=ru_{n+1}+su_n}$ for all ${n\ge 0}$, and with ${u_0=0}$, ${u_1=1}$. The parameters $r,~s$ are assumed to be nonzero coprime integers such that ${r^2+4s\ne 0}$. In this case, 
$$
u_n=\frac{\alpha^n-\beta^n}{\alpha-\beta}\quad {\text{\rm holds~for~all}}\quad n\ge 0,
$$
where $\alpha,\beta$ are the two roots of the quadratic ${x^2-rx-s=0}$. It is further assumed that $\alpha/\beta$ is not a root of unity. The Lucas sequences have nice divisibility properties. For example, if $m,~n$ are positive integers with ${m\mid n}$ then ${u_m\mid u_n}$. 

A primitive divisor of $u_n$ is a prime factor $p$ of $u_n$ which does not divide $u_m$ for any positive integer ${m<n}$ and does not divide ${r^2+4s}$. Working with the sequence of algebraic integers of general term ${v_n=(\alpha-\beta)u_n =\alpha^n-\beta^n}$, one can reformulate the above definition by saying that a primitive divisor is a prime number $p$ which divides $v_n$ but not $v_m$ for any positive integer ${m<n}$. It was shown in~\cite{BHV} that primitive divisors always exist if ${n\ge 31}$. Particular instances of this result were proved much earlier by Zsygmondy~\cite{Zi} (the case of rational integers $\alpha,~\beta$) and Carmichael~\cite{Car} (the case of real $\alpha,~\beta$). 

It is known that primitive divisors are congruent to $\pm 1\pmod n$. In particular, writing $P(m)$ for the largest prime factor of the integer $m$ with the convention that ${P(0)=P(\pm 1)=1}$, one has ${P(u_n)/n\ge (n-1)/n}$ for ${n\ge 31}$. Erd\H os~\cite{Erd} conjectured that $P(u_n)/n$ tends to infinity. This was proved to be so by Stewart~\cite{St13} who showed that $P(u_n)>n\exp(\log n/(104 \log\log n))$ holds for $n>n_0$, where $n_0$ is a constant which Stewart did not compute and which depends on the discriminant of the field ${\mathbb Q}(\alpha)$  and the number of distinct prime factors of $s$. Explicit values for $n_0$ were computed in~\cite{BHS21} 
at the cost of replacing $1/104$ by somewhat smaller constants (see Theorem~s~\ref{thordrat} and~\ref{thordquad} below). It is also shown in~\cite{BHS21} that~$n_0$ depends only on the field $\Q(\alpha)$, but is independent of the number of prime divisors of~$s$. 

 Schinzel~\cite{SC74} generalized the primitive divisor theorem to algebraic numbers in the following way. Let~$\gamma$ be an algebraic number of degree~$d$ which is not a root of unity, and denote ${v_n=\gamma^n-1}$. A  prime ideal ${\gerp\subset {\mathcal O}_{\mathbb K}}$ is called a primitive divisor of~$v_n$ if~$\gerp$ appears at positive exponent in the factorization of the principal fractional ideal $v_n{\mathcal O}_{\mathbb K}$  but~$\gerp$ does not appear in the factorization of $v_m{\mathcal O}_{\mathbb K}$ for any positive integer ${m<n}$.
 
 Schinzel proved that~$v_n$ has a primitive divisor for ${n\ge n_0(d)}$. Stewart~\cite{St77} gave an explicit value for $n_0(d)$ but he assumed that~$\gamma$ has a representation of the form ${\gamma=\alpha/\beta}$ with coprime integers $\alpha,\beta$ in ${\mathcal O}_{\mathbb K}$. An explicit value for $n_0$ without any additional hypothesis was given in~\cite{BL21}. 

In this note we show that Stewart's type result can be obtained for recurrent sequences other than Lucas. 
We look at the prime factors of a certain linear recurrent sequences of order~$4$ which is a particular instance of a norm of a complex quadratic Lucas sequence. Namely, we let~$q$ and~$a$ be  integers satisfying 
$$
q\ge 2, \qquad |a| < 2\sqrt q. 
$$
We denote~$\alpha$ and~$\bar\alpha$  the complex conjugate roots of ${x^2-ax+q}$. We prove the following theorem. 

\begin{theorem}
\label{thmain}
Set ${n_0:=\exp\exp(\max\{10^{10},3q\})}$
Let~$n$ be a positive integer satisfying  
${n\ge n_0}$. 
Then the rational integer ${(\alpha^n-1)(\bar\alpha^n-1)}$ has a prime divisor~$p$ satisfying 
\begin{equation*}
p\ge n \exp\left(0.0001\frac{\log n}{\log\log n}\right). 
\end{equation*}
\end{theorem}

When~$q$ is a prime power, the number 
$$
(\alpha-1)(\bar\alpha-1)=\alpha \bar\alpha-(\alpha+\bar\alpha)+1=q-a+1
$$ 
is the order of the group $\#E({\mathbb F}_q)$ of $\F_q$-rational points   on a certain elliptic curve~$E$.  Furthermore, ${(\alpha^n-1)(\bar\alpha^n-1)}$ represents the order of the group $\#E({\mathbb F}_{q^n})$ of  ${\mathbb F}_{q^n}$-rational points. 
The numbers $(\#E({\mathbb F}_{q^n}))_{n\ge 1}$ form a linearly recurrent sequence of order~$4$ 
with  roots $1,\alpha,\bar\alpha,q$. Like the Lucas sequences, these numbers have the property that $\#E(F_{q^m})\mid \#E(F_{q^n})$ when $m\mid n$ (because ${\mathbb F}_{q^n}$ is an extension of ${\mathbb F}_{q^m}$ of degree $n/m$). However, in spite of those similarities, some non-trivial new ideas  are needed to extend Stewart's argument to these sequences, see Subsection~\ref{ssfptwo}.

Note that  big prime factors  of orders of elliptic curves were studied before, albeit in a  different set-up. For instance, Akbary~\cite{Ak08} studied big prime factors of $\#E(\F_q)$, where~$E$ is a fixed  elliptic curve over~$\Q$ with complex multiplication.  He proved that, for a positive proportion of primes~$q$,  the number $\#E(\F_q)$ has a prime divisor bigger than $q^\theta$, where ${ \theta=1- e^{-1/4}/2=0.6105\ldots}$ We invite the reader to consult the comprehensive survey~\cite{Co19} for more information.

\subsection{Notation}
\label{ssnota}
Unless the contrary is stated explicitly,~$m$ and~$n$ (with or without indices) always denote positive integers and~$p$ (with or without indices) denotes a prime number.

Let~${\mathbb K}$ be a number field. We denote $D_{\mathbb K}$ and $h_{\mathbb K}$ the discriminant and the class number of~${\mathbb K}$. 
By a prime of~${\mathbb K}$ we mean a prime ideal of the ring of integers~$\calO_{\mathbb K}$.   If~$\gerp$ is prime of~${\mathbb K}$ with underlying rational prime~$p$, then we denote $f_\gerp$ its absolute residual degree and ${\norm\gerp =p^{f_\gerp}}$ its absolute norm.

We denote $\height(\alpha)$ the usual absolute logarithmic height of ${\alpha\in\bar\Q}$:
$$
\height(\alpha) = [{\mathbb K}:\Q]^{-1}\sum_{v\in M_{\mathbb K}}[{\mathbb K}_v:\Q_v]\log^+|\alpha|_v,
$$
where ${\log^+ =\max\{\log , 0\}}$. Here~${\mathbb K}$ is an arbitrary number field containing~$\alpha$, and the places ${v\in M_{\mathbb K}}$ are normalized to extend standard places of~$\Q$; that is, ${|p|_v=p^{-1}}$ if ${v\mid p <\infty}$ and ${|2021|_v=2021}$ if ${v\mid\infty}$. 
 
If~${\mathbb K}$ is a number field of degree~$d$ and ${\alpha\in {\mathbb K}}$ then the following formula is an immediate consequence of the definition of the height: 
$$
\height(\alpha) =\frac{1}{d}\left(\sum_{\sigma:{\mathbb K}\hookrightarrow\C}\log^+|\sigma(\alpha)|+\sum_\gerp\max\{0,-\nu_\gerp(\alpha)\}\log\norm\gerp\right), 
$$
where the first sum runs over the complex embeddings of~${\mathbb K}$ and the second sum runs over the primes of~${\mathbb K}$.   If ${\alpha \ne 0}$ then ${\height(\alpha) =\height(\alpha^{-1})}$, and we obtain the formula 
\begin{equation}
\label{ehe}
\height(\alpha) =\frac{1}{d}\left(\sum_{\sigma:{\mathbb K}\hookrightarrow\C}-\log^-|\alpha^{\sigma}|+\sum_\gerp\max\{0,\nu_\gerp(\alpha)\}\log\norm\gerp\right), 
\end{equation}
where ${\log^-=\min\{\log,0\}}$. 

Besides ${\log^+}$ and $\log^-$ we will also widely use  
$$
\logast =\max\{\log , 1\}.
$$

We use $O_1(\cdot)$ as the quantitative version of the familiar $O(\cdot)$ notation: ${A=O_1(B)}$ means ${|A|\le B}$. 

\section{Auxiliary facts}

\subsection{The Theorems of Stewart}

The following two theorems are, essentially, due to Stewart~\cite{St13}, though in the present form they can be found in~\cite{BHS21}, see Theorems~1.4 and~1.5 therein. 
 
\begin{theorem}
\label{thordrat}
Let~$\gamma$ be a non-zero algebraic number of degree~$d$, not a root of unity.  
Set
${p_0=\exp(80000 d (\logast d)^2)}$.  
Then for every prime~$\gerp$ of the field ${{\mathbb K}=\Q(\gamma)}$ whose absolute norm $\norm\gerp$ satisfies ${\norm\gerp\ge p_0}$, and every positive integer~$n$   we have 
$$
\nu_\gerp(\gamma^n-1) \le \norm\gerp\exp\left(-0.002d^{-1}\frac{\log \norm\gerp}{\log\log \norm\gerp}\right)\height(\gamma)\logast n. 
$$
\end{theorem}

\begin{theorem}
\label{thordquad}
Let~$\gamma$ be a non-zero algebraic number of degree~$2$, not a root of unity.  Assume that ${\norm\gamma=\pm1}$.  Set
$
{p_0=\exp\exp(\max\{10^8,2|D_K|\})}
$,
where $D_{\mathbb K}$ is the discriminant of the quadratic field ${{\mathbb K}=\Q(\gamma)}$. Then for every prime~$\gerp$ of~${\mathbb K}$ with underlying rational prime ${p\ge p_0}$,  and every positive integer~$n$   we have 
\begin{equation}
\label{eordquad}
\nu_\gerp(\gamma^n-1) \le p\exp\left(-0.001\frac{\log p}{\log\log p}\right)\height(\gamma)\logast n.
\end{equation}
\end{theorem}

\subsection{Cyclotomic polynomials and primitive divisors}
\label{sscypr}

Let~${\mathbb K}$ be a number field of degree~$d$ and ${\gamma\in {\mathbb K}^\times}$ not a root of unity. We consider the sequence ${u_n=\gamma^n-1}$. We call a ${\mathbb K}$-prime~$\gerp$ \textit{primitive divisor} of $u_n$ if 
$$
\nu_\gerp(u_n)\ge 1, \qquad \nu_\gerp(u_k)= 0 \quad (k=1,\ldots, n-1). 
$$
Let us recall some basic properties of primitive divisors. We denote by $\Phi_n(t)$ the $n$th cyclotomic polynomial.  

Items~\ref{ipd} and~\ref{idegtwo} of the following proposition are well-known and easy, and item~\ref{inpd} is Lemma~4 of Schinzel~\cite{SC74}; see also \cite[Lemma~4.5]{BL21}.

\begin{proposition}
\label{pprim}
\begin{enumerate}
\item
\label{ipd}
Let~$\gerp$ be a primitive divisor of $u_n$. Then  ${\nu_\gerp(\Phi_n(\gamma) )\ge 1}$ and ${\norm\gerp\equiv1\bmod n}$; in particular, ${\norm\gerp \ge n+1}$.

\item
\label{idegtwo}
Let~$\gerp$ be a primitive divisor of $u_n$ and~$p$ the rational prime underlying~$\gerp$. If~$\gamma$  is of degree~$2$ and absolute norm~$1$, then  ${p\equiv\pm1\bmod n}$. More specifically,  
$$
p\equiv 
\begin{cases}
1\bmod n & \text{if~$p$ splits  in~$\Q(\gamma)$}, \\
-1\bmod n &\text{if~$p$ is intert in~$\Q(\gamma)$}. 
\end{cases}
$$ 

\item
\label{inpd}
Assume that ${n\ge 2^{d+1}}$. Let~$\gerp$ be not a primitive divisor of~$u_n$. Then ${\nu_\gerp(\Phi_n(\gamma))\le \nu_\gerp(n)}$.
\end{enumerate}
\end{proposition}

\begin{remark}
In item~\eqref{idegtwo} the ramified~$p$ seem to be missing. However, it is easy to show that, when ${\norm\gamma=1}$ and~$p$ ramifies in $\Q(\gamma)$ then  ${\nu_\gerp(\gamma-1)>0}$ or ${\nu_\gerp(\gamma+1)>0}$. Hence, ${n=1}$ or ${n=2}$  in this case.  
\end{remark}

\subsection{Counting $S$-units}
Let~$S$ be a set of prime numbers. A positive integer is called $S$-unit if all its prime factors belong to~$S$. We denote $\Theta(x,S)$ the counting function for $S$-units:
$$
\Theta(x,S) =\#\{n\le x : p\mid n\Rightarrow p\in S\}. 
$$
We want to bound this function from above. 

\begin{proposition}
\label{psunits}
Let~$S$ be a set of~$k$ prime numbers.  Then for ${x\ge 3}$ we have 
\begin{equation}
\label{esunits}
\Theta(x,S) \le \exp\left(2k^{1/2}\log\log x+ 20\left(\frac{\log x}{\logast k}\right)\logast \left(\frac{k\logast k}{\log x}\right)\right).
\end{equation}
\end{proposition}

To start with, note the following trivial bound.

\begin{proposition}
\label{ptrivs}
In the set-up of Proposition~\ref{psunits} assuming $x\ge 7$ we have 
\begin{equation}
\label{etrivs}
\Theta(x,S) \le \exp(2k\log\log x). 
\end{equation}
\end{proposition}

\begin{proof}
If ${n\le x }$ then for every~$p$ we have ${\nu_p(n) \le \log x/\log 2}$. Hence 
$$
\Theta(x,S) \le \left(\frac{\log x}{\log 2}+1\right)^k \le \exp(2k\log\log x), 
$$
as wanted. 
\end{proof}

Next, let us consider a special case, when the primes from~$S$ are not too small.

\begin{proposition}
\label{psunitsbis}
In the set-up of Proposition~\ref{psunits}, assume  that  ${p\ge k^{1/2}}$ for every ${p\in S}$. Then 
\begin{equation}
\label{esunitsbis}
\Theta(x,S) \le \exp\left(10\left(\frac{\log x}{\logast k}\right)\logast \left(\frac{k\logast k}{\log x}\right)\right).
\end{equation}
\end{proposition}

\begin{proof}
If $x<7$, then either $\Theta(x,S)=0$ so the above inequality is trivially true, or $k\le 25$, and the right--hand side above is at least 
$$
\exp\left(\left(\frac{10}{\log 25}\right)\log x\right)>x^3>\lfloor x\rfloor\ge \Theta(x,S).
$$ 
If $x\ge 7$ and ${k\le 2}$ then~\eqref{esunitsbis} follows from~\eqref{etrivs}. From now on we assume that ${k\ge 3}$; in particular, ${\logast k=\log k}$. 
Write ${S= \{p_1 , p_2 ,\ldots , p_k\}}$.  Then every $S$-unit~$n$  can be presented as ${p_1^{a_1}\cdots p_k^{a_k}}$ with non-negative integers ${a_1, \ldots, a_k}$. If ${n\le x}$ then 
$$
a_1\log p_1 +\cdots +a_k\log p_k \le \log x. 
$$
By the assumption, ${\log p_i\ge (1/2)\log k}$ for ${i=1, \ldots, k}$. Hence,
\begin{equation}
\label{einequality}
a_1+\cdots+a_k  \le \ell, 
\end{equation}
where  ${\ell= \lfloor 2\log x/\log k\rfloor}$. 
We may assume that ${\ell\ge 1}$: if ${\ell=0}$ then the only solution of~\eqref{einequality} is ${a_1=\cdots=a_k=0}$,  and ${\Theta(x,S)=1}$. For further use, note that 
\begin{equation*}
\frac{\log x}{\log k}\le \ell \le 2\left(\frac{\log x}{\log k}\right). 
\end{equation*}
Inequality~\ref{einequality} has exactly 
$$
\sum _{i=0}^\ell \binom{k+i}{i} 
$$
solutions in ${(a_1, \ldots, a_k)\in \Z_{\ge0}^k}$. Hence, 
\begin{align*}
\Theta(x,S) &\le (\ell+1) \binom{k+\ell}{\ell} \\
&\le (\ell+1) \left(e\left(\frac{k+\ell}{\ell}\right)\right)^\ell\\
&\le  \exp\left(\ell \log \left(2e\left(\frac{k+\ell}{\ell}\right)\right)\right) \qquad\text{(we used ${\ell+1\le 2^\ell}$)} \\
&\le \exp\left(2\left(\frac{\log x}{\log k}\right)\log\left(2e\left(\frac{k+\ell}{\ell}\right)\right)\right)   . 
\end{align*}
If ${k\le 9\ell}$ then  
$$
\log\left(2e\frac{k+\ell}{\ell}\right) \le \log(20e) <4, 
$$
and we are done. If ${k\ge 9\ell}$ then 
$$
\log\left(2e\left(\frac{k+\ell}{\ell}\right)\right) \le \log\left(8 \left(\frac{k}{\ell}\right)\right)\le \log \left(8\left(\frac{k\log k}{\log x}\right)\right) \le 4 \logast \left(\frac{k\log k}{\log x}\right), 
$$
and we are done again. 
\end{proof}

\begin{proof}[Proof of Proposition~\ref{psunits}]
Write ${S=S_1\cup S_2}$, where 
$$
S_1=\{p\in S :p <k^{1/2}\}, \qquad S_2=\{p\in S :p \ge k^{1/2}\}. 
$$
Then, clearly ${\Theta(x,S) \le \Theta (x, S_1)\Theta(x,S_2) }$.  We estimate $\Theta (x, S_1)$ using Proposition~\ref{ptrivs} and $\Theta(x,S_2)$ using Proposition~\ref{psunitsbis}:
\begin{align*}
\Theta (x, S_1)&\le  \exp(2k^{1/2}\log\log x), \\ 
\Theta(x,S_2)&\le 
\exp\left(10\left(\frac{\log x}{\logast (k-k^{1/2})}\right)\logast \left(\frac{k\logast k}{\log x}\right)\right)\\
 &\le \exp\left(20\left(\frac{\log x}{\logast k}\right)\logast \left(\frac{k\logast k}{\log x}\right)\right).
\end{align*}
The result follows. 
\end{proof}

\section{Proof of Theorem~\ref{thmain}}

Denote ${{\mathbb K}=\Q(\alpha)}$. It is an imaginary quadratic field. Hence, for a non-zero ${\theta\in {\mathcal O}_{\mathbb K}}$ we have  
$$
\height(\theta) = \log|\theta|=\frac12 \sum_\gerp\nu_\gerp(\theta) \log \norm\gerp, 
$$
the sum being over the finite primes of~${\mathbb K}$.

We apply this with ${\theta=\Phi_n(\alpha)}$ (recall that $\Phi_n(t)$ denotes the $n$th cyclotomic polynomial). We have 
\begin{equation}
\label{elogphin}
\log|\Phi_n(\alpha)| = \ph(n)\log|\alpha|+ \sum_{d\mid n}\mu\left(\frac nd\right)\log|1-\alpha^{-d}| = \frac12\ph(n)\log q+ O_1(5). 
\end{equation}
Indeed, we have ${|\alpha|=q^{1/2}\ge \sqrt2}$ and 
${\bigl|\log|1+z|\bigr| \le 2|z|}$ for ${|z|\le 1/\sqrt2}$. Hence
$$
\left|
\sum_{d\mid n}\mu\left(\frac nd\right)\log|1-\alpha^{-d}|
\right| < 2\sum_{d=1}^\infty |\alpha|^{-d} <5,  
$$
which proves~\eqref{elogphin}. Thus,
$$
\sum_\gerp\nu_\gerp(\Phi_n(\alpha)) \log \norm\gerp = \ph(n)\log q+ O_1(10). 
$$
Proposition~\ref{pprim}.\ref{inpd} implies that, for ${n\ge 8}$, 
$$
\sum_{\text{$\gerp$ not primitive}} \nu_\gerp(\Phi_n(\alpha)) \log \norm\gerp \le 2\log n, 
$$
the sum being over~$\gerp$ which are non-primitive divisors of ${\alpha^n-1}$.  Hence,
$$
\sum_{\text{$\gerp$ primitive}} \nu_\gerp(\Phi_n(\alpha)) \log \norm\gerp \ge \ph(n)\log q  -10-2\log n. 
$$
The Euler totient function $\ph(n)$ satisfies 
\begin{equation}
\label{ersph} 
\ph(n) \ge 0.5 \frac{n}{\log\log n} \qquad (n\ge 10^{20})
\end{equation}
(see~\cite[Theorem~15]{RS62}). 
Hence for ${n\ge 10^{20}}$ we have
$$
\sum_{\text{$\gerp$ primitive}} \nu_\gerp(\Phi_n(\alpha)) \log \norm\gerp \ge 0.8 \ph(n)\log q. 
$$
From now on, the proof splits into two cases, depending on whether the primes with residual degree~$1$ contribute more to the sum, or those with residual degree~$2$ do. Precisely, we have
\begin{align}
\label{efpone}
\text{either}\qquad \sum_{\genfrac{}{}{0pt}{}{\text{$\gerp$ primitive}}{f_\gerp=1}} \nu_\gerp(\Phi_n(\alpha)) \log \norm\gerp &\ge 0.4 \ph(n)\log q, \\
\label{efptwo}
\text{or}\qquad \sum_{\genfrac{}{}{0pt}{}{\text{$\gerp$ primitive}}{f_\gerp=2}} \nu_\gerp(\Phi_n(\alpha)) \log \norm\gerp &\ge 0.4 \ph(n)\log q. 
\end{align} 
Case~\eqref{efpone} is easier, the proof follows the same lines as the proof of Theorem~1.2 in~\cite{BHS21}. Case~\eqref{efptwo} is harder and requires more intricate arguments. 

\subsection{Case~\eqref{efpone}}
\label{sseasy}

{\sloppy 

We will apply  Theorem~\ref{thordrat} with ${\gamma=\alpha}$ and ${{\mathbb K}=\Q(\alpha)}$, so that ${d=2}$ and ${p_0=\exp(160000)}$. We may assume that ${n>p_0}$, because~$n_0$ from Theorem~\ref{thmain} is bigger than~$p_0$. 

}

Let~$P$ be the biggest rational prime~$p$ with the following two properties:~$p$ splits in ${{\mathbb K}=\Q(\alpha)}$, and ${\alpha^n-1}$ admits a primitive divisor~$\gerp$ with underlying prime~$p$. We want to show that 
\begin{equation}
\label{ewant}
P>n\exp\left(0.0002\frac{\log n}{\log\log n}\right). 
\end{equation}
Let~$\gerp$ be a primitive divisor of ${\alpha^n-1}$ with ${f_\gerp=1}$, and~$p$ the underlying rational prime. Then ${p\le P}$ and 
${p=\norm\gerp \equiv1\bmod n}$ by Proposition~\ref{pprim}.\ref{ipd}. In particular, ${p>n>p_0}$, and Theorem~\ref{thordrat} applies:
\begin{align*}
\nu_\gerp(\alpha^n-1) &\le p\exp\left(-0.001\frac{\log p}{\log\log p}\right)\cdot \frac 12 \log q \log n \\
&\le P\exp\left(-0.001\frac{\log n}{\log\log n}\right)  \log q \log n. 
\end{align*}
Hence, 
$$
\sum_{\genfrac{}{}{0pt}{}{\text{$\gerp$ primitive}}{f_\gerp=1}} \nu_\gerp(\Phi_n(\alpha)) \log \norm\gerp \le \pi(P; n,1)P\exp\left(-0.001\frac{\log n}{\log\log n}\right)  \log q \log n \log P ,  
$$
where, as usual ${\pi(x;m,a)}$ counts prime in the residue class ${a\bmod m}$. 
Estimating trivially ${\pi(P; n, 1)\le P/n}$,  we obtain 
$$
\sum_{\genfrac{}{}{0pt}{}{\text{$\gerp$ primitive}}{f_\gerp=1}} \nu_\gerp(\Phi_n(\alpha)) \log \norm\gerp \le \frac{P^2\log P}{n} \exp\left(-0.001\frac{\log n}{\log\log n}\right) \log n  \log q .   
$$
Compared with~\eqref{efpone}, this implies 
$$
P^2\log P \ge 0.4 \frac{n\ph(n)}{\log n} \exp\left(0.001\frac{\log n}{\log\log n}\right). 
$$
Using~\eqref{ersph}, this implies~\eqref{ewant} for ${n>n_0}$.

\subsection{Case~\eqref{efptwo}}
\label{ssfptwo}
If~$\gerp$ is a prime of~${\mathbb K}$ with ${f_\gerp=2}$ then it is a rational prime, and we write~$p$ instead of~$\gerp$. For such~$p$ we have ${\nu_p(\alpha^n-1)=\nu_p(\bar\alpha^n-1)}$. Setting ${\gamma =\bar\alpha/\alpha}$, we obtain
$$
\nu_p(\gamma^n-1) \ge\nu_p\bigl((\bar\alpha^n-1)-(\alpha^n-1)\bigr) \ge \nu_p(\alpha^n-1)\ge \nu_p(\Phi_n(\alpha)). 
$$
Hence,~\eqref{efptwo} implies the inequality 
\begin{equation*}
\sum_{p\in \calP}\nu_p(\gamma^n-1) \log p \ge 0.2\ph(n)\log q
\end{equation*} 
(note that ${\norm p =p^2}$), where the set~$\calP$ consists of the rational primes~$p$ inert in~$K$ and satisfying ${\nu_p(\alpha^n-1)>0}$:
$$
\calP=\{\text{$p$ inert in~${\mathbb K}$ and  ${\nu_p(\alpha^n-1)>0}$}\}. 
$$
We are now tempted to bound the sum on the left as we did in Subsection~\ref{sseasy}, but with Theorem~\ref{thordrat} replaced by Theorem~\ref{thordquad}, which applies here because ${\norm\gamma=1}$. However, now instead of ${p\equiv 1\bmod n}$ we have merely ${p^2\equiv 1\bmod n}$, and    we have to use a more delicate argument.

Denote ${v_n=\gamma^n-1}$. If ${\nu_p(v_n)>0}$ then there is a divisor~$d$ of~$n$ such that~$p$ is primitive for $v_{n/d}$. We denote it~$d_p$. We have
$$
\nu_p(v_n) \le  \nu_p(v_{n/d_p}) + \sum_{\genfrac{}{}{0pt}{}{m\mid n}{m\ne n/d_p}}\nu_p(\Phi_m(\gamma)). 
$$
Proposition~\ref{pprim}.\ref{inpd} bounds the sum on the right by 
$$
\sum_{m\mid n}\nu_p(m) +\sum_{m=1}^7\nu_p(\Phi_m(\gamma)). 
$$
It follows that
\begin{equation*}
\sum_{p\in \calP}\nu_p(\gamma^n-1) \log p \le \sum_{p\in \calP} \nu_p(v_{n/d_p})  + \sum_{m\mid n}\log m + \sum_{m=1}^7\sum_p\nu_p(\Phi_m(\gamma))\log p. 
\end{equation*}
The middle sum on the right is trivially estimated by ${\tau(n)\log n}$, where $\tau(n)$ denotes the number of divisors of~$n$:
$$
\tau(n)=\sum_{m\mid n}1. 
$$
To estimate the double sum on the right, note that 
\begin{equation*}
\nu_p(\Phi_m(\gamma)) \le \nu_p(v_m) \le \frac 12 \nu_p\bigl((\alpha^m-\bar\alpha^m)^2\bigr). 
\end{equation*}
Since ${(\alpha^m-\bar\alpha^m)^2}$ is a rational integer of absolute value not exceeding $4q^m$, this implies that 
\begin{equation}
\label{edivides}
\sum_p \nu_p(v_m)\log p \le \frac 12 m\log q +\log 2.  
\end{equation}
Hence,
$$
\sum_{m=1}^7\sum_p\nu_p(\Phi_m(\gamma))\log p \le 14\log q + 7\log2. 
$$
Putting all this together, we obtain the inequality 
\begin{equation*}
\sum_{p\in \calP} \nu_p(v_{n/d_p})\log p \ge 0.2\ph(n)\log q - \tau(n)\log n - 14\log q-7\log 2. 
\end{equation*}

\subsubsection{Disposing of big~$d_p$}

We want to get rid in our sum of primes~$p$ with ${d_p\ge \tau(n)\log n }$.  %Fix a divisor~$d$ of~$n$. 
Using~\eqref{edivides}, we obtain 
$$
\sum_{d_p\ge \tau(n)\log n } \nu_p(v_{n/d_p})\log p \le \frac12n\log q \sum_{\genfrac{}{}{0pt}{}{d\mid n}{d\ge \tau(n)\log n }}\frac1d + \tau(n)\log 2 
$$
The sum on the right is trivially estimated as
$$
\frac{\tau(n)}{\tau(n)\log n}=\frac{1}{\log n}. 
$$
Hence ,
$$
\sum_{d_p\ge \tau(n)\log n } \nu_p(v_{n/d_p})\log p \le \frac{n}{2\log n}\log q  + \tau(n)\log 2 . 
$$
Denote by~$\calP'$ the subset of~$\calP$ consisting of~$p$ with  ${d_p < \tau(n)\log n}$:
$$
\calP'=\{ p \in \calP: \ d_p < \tau(n)\log n\}. 
$$
Then we obtain 
\begin{align*}
\sum_{p\in \calP'} \nu_p(v_{n/d_p})\log p &\ge  0.2\ph(n)\log q - \tau(n)\log n - 14\log q-7\log 2\\
&\hphantom{\ge}-\frac{n}{2\log n}\log q  - \tau(n)\log 2 . 
\end{align*}
We have 
\begin{equation}
\label{etaun}
\tau(n) \le \exp\left(1.1\frac{\log n}{\log\log n}\right)\qquad (n\ge 3)
\end{equation}
(see \cite[Theorem~1]{NR83}). Using this and~\eqref{ersph}, we deduce that, for 
$$
n\ge n_0\ge \exp\exp(10^{10})
$$
(which is true by assumption), we have 
\begin{equation}
\label{esumpprime}
\sum_{p\in\calP'} \nu_p(v_{n/d_p})\log p \ge  0.1\ph(n)\log q. 
\end{equation}

\subsubsection{Counting divisors ${d<\tau(n)\log n}$}
The number of divisors ${d<\tau(n)\log n}$  can be estimated using Proposition~\ref{psunits}. Denote ${x=\tau(n)\log n}$ and denote by~$S$ the set of prime factors of~$n$, so that ${\# S =\omega(n)}$.   Then 
\begin{align*}
\#\{d\mid n : d<x\} &\le \Theta(x,S)\\
& \le \exp\left(2\omega(n)^{1/2}\log\log x+ 20\frac{\log x}{\logast \omega(n)}\logast \frac{\omega(n)\logast \omega(n)}{\log x}\right). 
\end{align*}
For further use, note the trivial estimates 
\begin{align}
\label{etaugeom}
\log \tau(n) & \ge \omega(n) \log 2, \\
\label{etauleom}
\log \tau(n) &\le \omega(n) \log \left(\frac{\log n}{\log2}+1\right) \le 2\omega(n)\log\log n
\end{align}
(recall that ${n\ge \exp\exp(10^{10})}$). Note also the estimates
\begin{align}
\label{elogtaun}
\log \tau(n) &\le 1.1\frac{\log n}{\log\log n}, \\
\label{eomegan}
\omega(n) &\le 1.4 \frac{\log n}{\log\log n}
\end{align}
(see~\eqref{etaun} and \cite[Théorème~11]{Ro83}). 

Using~\eqref{elogtaun} and~\eqref{eomegan}, we deduce that, for ${n\ge \exp\exp(10^{10})}$, we have 
\begin{equation}
2\omega(n)^{1/2}\log\log x \le (\log n)^{1/2}\log\log n. 
\end{equation}
Using~\eqref{etaugeom} and~\eqref{eomegan}, we deduce that 
\begin{equation}
\frac{\omega(n)\logast \omega(n)}{\log x} \le \frac{\omega(n)\logast \omega(n)}{\log \tau(n)} \le \frac{\logast \omega(n)}{\log 2} \le 2\log\log n. 
\end{equation}
To estimate ${\log x/\logast\omega(n) }$, we consider two cases. Assume first that 
$$
\omega(n)\le \frac{\log n}{(\log\log n)^3}. 
$$
In this case, using~\eqref{etauleom}, we estimate 
$$
\frac{\log x}{\logast\omega(n)} \le \frac{2\omega(n) \log \log n+\log\log n}{1} \le 3\omega(n) \log \log n\le  3 \frac{\log n}{(\log\log n)^2}. 
$$
Now assume that 
$$
\omega(n)\ge \frac{\log n}{(\log\log n)^3}. 
$$
In this case, using~\eqref{elogtaun}, we obtain
$$
\frac{\log x}{\logast\omega(n)} \le \frac{1.1\frac{\log n}{\log\log n}+\log\log n}{\log\log n-3\log \log \log n} \le  3 \frac{\log n}{(\log\log n)^2}. 
$$
Thus, in any case
$$
\frac{\log x}{\logast\omega(n)} \le  3 \frac{\log n}{(\log\log n)^2}. 
$$
Putting this all together, we obtain
\begin{align}
\#\{d\mid n : d<x\} & \le \exp\left((\log n)^{1/2}\log\log n+ 20\cdot 3 \frac{\log n}{(\log\log n)^2}\log (2\log\log n) \right) \nonumber\\
\label{emanylogs}
&\le \exp\left( 70\frac{\log n\log\log\log n}{(\log\log n)^2} \right). 
\end{align}

\subsubsection{The cardinality of~$\calP'$}

The crucial step is estimating the number of primes in the set~$\calP'$. Denote~$P$ the biggest element of~$\calP'$. 
We are going to prove that 
\begin{equation}
\label{eboundpprime}
\#\calP' \le \left(\frac Pn +1\right) \exp\left( 80\frac{\log n\log\log\log n}{(\log\log n)^2} \right).
\end{equation}

Let~$p$ be a prime from the set~$\calP'$. Recall that ${n\mid p^2-1}$; in particular, ${p>2}$. Assume first that~$n$ is odd. In this case the numbers 
${\gcd(p-1, n)}$  and ${\gcd(p+1,n)}$ 
are coprime. We write them, respectively,~$d$ and~$n/d$. 
Thus, we have
\begin{equation}
\label{endd}
p\equiv -1\bmod n/d, \qquad p\equiv 1\bmod d
\end{equation}
for some~$d$ dividing~$n$ and such that ${\gcd(n/d,d)=1}$. By the definition of~$d_p$ we must have ${d\mid d_p}$. In particular, if ${p\in \calP'}$ then ${d< \tau(n)\log n}$. 

By the Chinese Remainder Theorem, for every ${d\mid n}$ such that ${\gcd(n/d,d)=1}$, there exists a unique ${a_d \in \{1, \ldots , n-1\}}$ such that ${p\equiv a_d\bmod n}$ holds for every~$p$ satisfying~\eqref{endd}. It follows that 
$$
\#\calP' \le \sum_{\genfrac{}{}{0pt}{}{d\mid n}{d<\tau(n)\log n}}\pi(P;n, a_d). 
$$
We estimate trivially ${\pi(P;n, a_d)\le P/n+1}$. Hence, when~$n$ is odd, we have the upper bound
\begin{equation}
\label{enodd}
\#\calP' \le \left(\frac{P}{n}+1\right)\#\{d\mid n : d<\tau(n)\log n\}. 
\end{equation}

If~$n$ is even, the argument is similar, but slightly more complicated. Assume, for instance, that ${p\equiv 3 \bmod 4}$. Then the numbers 
$$
\gcd\left(\frac{p-1}{2}, \frac{n}{2}\right) , \qquad \gcd\left(p+1, \frac{n}{2}\right) 
$$
are coprime, and we write them~$d$ and $n/2d$, respectively; note also that~$d$ is odd. We have ${2d\mid d_p}$, and, in particular, ${d<\tau(n)\log n}$.  The system of congruences 
\begin{equation*}
p\equiv -1\bmod \frac{n}{2d}, \qquad p\equiv 1\bmod d
\end{equation*}
is equivalent to  ${p\equiv a_d\bmod n/2}$, where ${a_d\in \{1, \ldots, n/2-1\}}$ depends only on~$d$. 
Similarly, when ${p\equiv 1 \bmod 4}$, we have ${p\equiv b_d\bmod n/2}$, where ${d<\tau(n)\log n}$ and ${b_d\in \{1, \ldots, n/2-1\}}$ depends only on~$d$. We obtain
\begin{align}
\#\calP' &\le \sum_{\genfrac{}{}{0pt}{}{d\mid n}{d<\tau(n)\log n}}\bigl(\pi(P;n/2, a_d)+\pi(P;n/2, b_d)\bigr)\nonumber\\
\label{eneven}
&\le \left(4\frac{P}{n}+2\right)\#\{d\mid n : d<\tau(n)\log n\}.
\end{align}
We see that upper bound~\eqref{eneven} holds in all cases. Combining it with~\eqref{emanylogs}, we obtain
$$
\#\calP' \le \left(\frac Pn +\frac12\right) \exp\left( 70\frac{\log n\log\log\log n}{(\log\log n)^2}+\log 4 \right), 
$$
which is sharper than~\eqref{eboundpprime}. 

\subsubsection{Using Stewart}
Now it is the time to use Theorem~\ref{thordquad}. To start with, note  that ${|D_{\mathbb K}|\le q}$. Hence,~$p_0$ from Theorem~\ref{thordquad} does not exceed ${n_0^{1/2}}$.  Now if ${\nu_p(\gamma^n-1)>0}$ then ${n\mid p^2-1}$, see Proposition~\ref{pprim}.\ref{ipd}. Hence, ${p>n^{1/2}\ge n_0^{1/2}\ge p_0}$, and Theorem~\ref{thordquad} applies. For ${p\in \calP'}$ it gives 
\begin{align}
\nu_p(\gamma^n-1) &\le p\exp\left(-0.001\frac{\log p}{\log\log p}\right)\height(\gamma)\log n \nonumber\\
\label{enuple}
&\le 2P\exp\left(-0.0005\frac{\log n}{\log\log n}\right)\log q\log n, 
\end{align}
because 
$$
p\le P, \qquad \frac{\log p}{\log\log p} \ge \frac12 \frac{\log n}{\log\log n}, \qquad \height(\gamma) \le 2q. 
$$
Since ${\nu_p(v_{n/d_p}) \le \nu_p(\gamma^n-1)}$, we can combine~\eqref{enuple} with~\eqref{esumpprime}, obtaining 
$$
2P\log P\exp\left(-0.0005\frac{\log n}{\log\log n}\right)\#\calP'\log q\log n 
\ge 0.1 \ph(n) \log q. 
$$ 
Using~\eqref{eboundpprime} and~\eqref{ersph}, this implies, for ${n\ge \exp\exp(10^{10})}$, that 
\begin{align*}
P(P+n)\log P &\ge n^2 \exp\left(\left(0.0004-100\frac{\log\log\log n}{\log \log n}\right)\frac{\log n}{\log\log n}\right)\\
&\ge n^2 \exp\left(0.0003\frac{\log n}{\log\log n}\right). 
\end{align*}
If ${P<n}$ then the latter inequality is clearly impossible for ${n\ge \exp\exp(10^{10})}$. Hence, ${P\ge n}$, and we obtain 
$$
P^2\log P \ge \frac12 n^2 \exp\left(0.0003\frac{\log n}{\log\log n}\right),
$$
which implies 
$$
P\ge n \exp\left(0.0001\frac{\log n}{\log\log n}\right). 
$$
Theorem~\ref{thmain} is proved. 

{\footnotesize

\bibliographystyle{amsplain}
\bibliography{ellip.bib}

\paragraph{Yuri Bilu \& Haojie Hong:} Institut de Mathématiques de Bordeaux, Université de Bordeaux \& CNRS, Talence, France

\paragraph{Florian Luca:}
School of Maths, Wits University, South Africa and
King Abdulaziz University, Jeddah, Saudi Arabia and
IMB, Universit\'e de Bordeaux, France and
Centro de Ciencias Matematicas UNAM, Morelia, Mexico

}

\end{document}